\documentclass[12pt]{amsart}
\usepackage{amssymb,amsmath,amsfonts,amsopn,url,bbm,amscd,enumerate}
\newtheorem{thm}{Theorem}[section]
\newtheorem{prop}[thm]{Proposition}
\newtheorem{lemma}[thm]{Lemma}
\newtheorem{cor}[thm]{Corollary}

\theoremstyle{definition}
\newtheorem{defn}[thm]{Definition}

\newcommand{\field}[1]{\mathbbm{#1}}
\newcommand{\N}{\field{N}}

\newcommand{\R}{\field{R}}
\newcommand{\C}{\field{C}}
\newcommand{\ideal}[1]{\mathfrak{#1}}
\newcommand{\m}{\ideal{m}}

\newcommand{\p}{\ideal{p}}
\newcommand{\q}{\ideal{q}}

\newcommand{\func}[1]{\mathrm{#1} \,}

\newcommand{\intcl}[1]{{#1}^{-}}
\newcommand{\sptc}[1]{{#1}^{*\rm sp}}
\newcommand{\spintcl}[1]{{#1}^{-\rm sp}}

\newcommand{\spF}[1]{{#1}^{F\rm sp}}
\newcommand{\spc}[2]{{#2}^{{#1} \rm sp}}

\newcommand{\ra}{\rightarrow}
\newcommand{\M}{\mathcal{M}}
\begin{document}

\title{Reductions and special parts of closures}

\author{Neil Epstein}
\email{nepstein@uos.de}
\address{Universit\"at Osnabr\"uck,
 Institut f\"ur Mathematik,
 49069 Osnabr\"uck, Germany}

\begin{abstract}
We provide an axiomatic framework for working with a wide variety of closure operations on ideals and submodules in commutative algebra, including notions of reduction, independence, spread, and special parts of closures.  This framework is applied to tight, Frobenius, and integral closures.  Applications are given to evolutions and special Brian\c{c}on-Skoda theorems.
\end{abstract}

\date{\today}
\maketitle

\section{Introduction}
One of the most useful notions in commutative algebra is that of the closure operation on ideals, or more generally on submodules.  Much of the time, authors concentrate on the properties of one particular closure operation, so the general notion itself is not always given a proper definition\footnote{One counterexample is the recent paper \cite{Va-cl}, which gives structure to sets of closure operations satisfying certain properties.}.  The following encapsulates what most authors mean:

\begin{defn}
Let $R$ be a ring, and let $\mathcal{M}$ be a set of $R$-modules (often either just $R$, or all (finitely generated) $R$-modules).  A \emph{closure operation} $c$ sends any submodule $L$ of a module $M \in \mathcal{M}$ to another submodule $L^c_M$ of $M$, subject to the following axioms \begin{enumerate}
\item For any submodule $L$ of any $M \in \mathcal{M}$, $L \subseteq L^c_M = (L^c_M)^c_M$
\item If $K \subseteq L$ are submodules of some $M \in \mathcal{M}$, then $K^c_M \subseteq L^c_M$.
\item Let $g: M \ra M'$ be a homomorphism of $R$-modules in $\mathcal{M}$.  Then for any submodule $L \subseteq M$, $g(L^c_M) \subseteq g(L)^c_{M'}$.
\end{enumerate}
\end{defn}

Familiar examples of closure operations on ideals include tight closure, integral closure, and the radical.  The definition above is very broad, so it is useful to identify additional properties that may not hold for all closure operations.  For instance, in \cite{nme*spread}, we introduced the following notions for closure operations on ideals, here generalized to the module case:
\begin{defn}
A closure operation $c$ on a class $\mathcal{M}$ of $R$-modules is \emph{Nakayama} if $(R,\m)$ is local and for every $M\in \mathcal{M}$ and submodules $K \subseteq L \subseteq M$ such that $K \subseteq L \subseteq (K + \m L)^c_M$, we have $K^c_M = L^c_M$.

Given $M \in \M$ and elements $z_1, \dotsc, z_t \in M$, they are \emph{$c$-independent} (relative to $M$) if for all $1\leq i \leq t$, we have \[
z_i \notin (\sum_{j \neq i} R z_j)^c_M.
\]
A submodule $L \subseteq M$ is \emph{$c$-independent} if it has a $c$-independent generating set, and it is \emph{strongly $c$-independent} if every minimal generating set for $L$ is $c$-independent.

Given a submodule $L \subseteq M \in \M$, a \emph{$c$-reduction} $K \subseteq L$ \emph{of $L$ in $M$} is a submodule such that $K^c_M = L^c_M$.  It is \emph{minimal} if there is no proper submodule $P \subsetneq K$ which is a $c$-reduction of $L$.  If minimal $c$-reductions exist, and if every minimal $c$-reduction of $L$ in $M$ has the same minimal number of generators, we call this common number the \emph{$c$-spread of $L$ in $M$}, denoted $\ell^c_M(L)$.
\end{defn}

As a closure on ideals in a local ring, it is clear that integral closure is Nakayama, and we showed in \cite{nme*spread} that tight closure in this context is as well.  Radical, however, is not Nakayama.  Moreover, we showed that for any Nakayama closure $c$, minimal $c$-reductions exist, and they are exactly the strongly $c$-independent $c$-reductions.  Although this result was stated for ideals, the exact same proof shows it to be true in this wider context.

Also in that paper, we used Vraciu's notion of \emph{special tight closure} to prove that under mild conditions on the ring, minimal generating sets of minimal $*$-reductions of an ideal all have the same size generating sets.

Accordingly, in this note we generalize and axiomatize the notion of the \emph{special part of a closure}, and we use it to obtain interesting results for Frobenius, integral, and tight closures. \emph{Except where otherwise noted, in this paper $R$ will always denote a Noetherian local ring with maximal ideal $\m$ and residue field $k$.}

Most of this work was completed years ago as part of my dissertation.  I did not submit the paper for publication at the time.  However, as there are now several papers which use the ideas in the paper (e.g. \cite{Vr-chains}, \cite{EH-cc}, and \cite{FoVa-core}), I have been convinced to publish it.

\section{Axioms for special parts of closures}
\begin{defn}\label{def:csp}
Let $c$ be a closure operation on submodules of an $R$-module $M$.
Then $c\rm sp$ is a \emph{special part of $c$ for $M$} if the following four axioms hold
whenever $L$ and $N$ are submodules of $M$.
\begin{enumerate}
\item $\spc{c}{L}_M$ is a submodule of $M$.
\item\label{ax:contmL} $\m L \subseteq \spc{c}{L}_M \subseteq L^c$.
\item\label{ax:commute} $\spc{c}{(L^c_M)}_M = \spc{c}{L}_M = (\spc{c}{L}_M)^c_M$.
\item\label{ax:Nak} If $L \subseteq N \subseteq (L + \spc{c}{N}_M)^c_M$, then $N \subseteq L^c_M$.
\end{enumerate}
\end{defn}
If $M=R$, we say $c\rm sp$ is a \emph{special part of $c$ for ideals}.  If the closure operation
$c$ is only defined for ideals, we simply say $c \rm sp$ is a \emph{special part of $c$}.  If
$c$ and $c\rm sp$ are defined at least on all submodules of finitely generated $R$-modules,
we say that $c\rm sp$ is a \emph{special part of $c$}.

If the ambient module is understood, sometimes we write $L^c$ in place of $L^c_M$.  In particular
if $M=R$ (so $L$ is now an ideal), we almost always leave off the ambient module $R$ in the notation.

Note the following consequences of the definition:

\begin{lemma}\label{lem:capmL}
Let $c$ be a closure operation on submodules of $M$ and $c\rm sp$ a special part of $c$.  Then
\begin{itemize}
\item $c$ is a Nakayama closure.
\item\label{ax:functorial} If $L^c_M \subseteq N^c_M$, then $\spc{c}{L}_M \subseteq \spc{c}{N}_M$.
\item\label{ax:capmL} For any $c$-independent submodule $L$, $\m L = L \cap \spc{c}{L}_M$.
\end{itemize}
\end{lemma}

\begin{proof}
$c$ is Nakayama because of axioms (\ref{ax:Nak}) and (\ref{ax:contmL}) of Definition~\ref{def:csp}.
For this reason, we call axiom (\ref{ax:Nak}) the \emph{Nakayama property}.

If $L^c_M \subseteq N^c_M$, then by axiom (\ref{ax:commute}) of Definition~\ref{def:csp},
$\spc{c}{L}_M = \spc{c}{(L^c)}_M \subseteq \spc{c}{(N^c)}_M = \spc{c}{N}_M$.

Finally, suppose that $L$ is a $c$-independent submodule of $M$, and let $z \in L \cap \spc{c}{L}_M$.
Let $z_1, \dotsc, z_n$ be a $c$-independent generating set of $L$.  Then there are 
elements $r_j \in R$ such that $z = \sum_{j=1}^n r_j z_j$.  If $z \notin \m L$, then 
there is some $j$ with $r_j \notin \m$.  Without loss of generality, $j=1$, and by dividing
by $r_1$, we may assume that $r_1 = 1$.  That is, \[
z = z_1 + \sum_{j=2}^n r_j z_j \in \spc{c}{L}_M.
\]
Let $N = (z_2, \dotsc, z_n)$.  Then by the above equation, we have $z_1 \in N + \spc{c}{L}_M$,
which implies that $L \subseteq N + \spc{c}{L}_M$.  Then by axiom (\ref{ax:Nak}), $L \subseteq N^c_M$,
so that $z_1 \in N^c_M$.  But this contradicts the $c$-independence of $z_1, \dotsc, z_n$.

Thus, $z \in \m L$.
\end{proof}

\section{The special part of tight closure}

The ideal case of the special part of tight closure was introduced by Vraciu in~\cite{Vr*ind}.  Further work appears in~\cite{nme*spread} and~\cite{Vr-chains}.  Here's the submodule version:

\begin{defn}
For finitely generated $R$-modules $N \subseteq M$, we define the \emph{special part of the tight closure}
of $N$ in $M$ to be the set \[ \sptc{N}_M = \{z \in M \mid \exists q \text{ such that } z^{q} \in
(\m N^{[q]}_M)^*_{F^{e}(M)} \}
\]
\end{defn}

Most of the proof that $*\rm sp$ is a special part of $*$ for ideals is in~\cite{nme*spread}, and
the proofs for the submodule case are identical.

Note also that the special part of tight closure can be computed modulo minimal primes.

In~\cite{Vr-chains}, Vraciu introduces the notions of $*$-independence modulo an ideal, $*$-spread modulo an ideal, and (minimal) $*$-reductions of an ideal modulo another ideal.  Note that $*$-independence modulo $J$ is just $*$-independence in the $R$-module $R/J$, (minimal) $*$-reductions of $I$ modulo $J$ correspond exactly with (minimal) $*$-reductions of $I/J$ in the module $R/J$, and $\ell^*_J(I) = \ell^*_{R/J}(I/J)$ whenever such a number is defined.  She observes that the proof of \cite[Theorem 5.1]{nme*spread} can be ``modified slightly'' to show that $\ell^*_J(I)$ exists in her context of a normal local domain.  Essentially the same modification shows that whenever $R$ is excellent and analytically irreducible and $L \subseteq M$ is any inclusion of finitely generated modules, then $\ell^*_M(L)$ exists.

 \section{Analytic $F$-independence, and the special part of Frobenius closure}
In this section, we assume only that $R$ is a Noetherian local ring of prime
characteristic $p>0$.	

\begin{defn}
Let $N \subseteq M$ be finitely generated $R$-modules.  The \emph{special part of the
Frobenius closure of $N$ in $M$} is the submodule \[
N^{F\rm sp}_M := \{z \in M \mid \exists q=p^e \text{ such that } z^q \in \m N^{[q]}_M \}.
\]
\end{defn}
It is equivalent to say that there is some $q$ such that $z^q \in (\m N^{[q]}_M)^F_M$.

\begin{prop} $F\rm sp$ is a special part of the Frobenius closure, in the sense
of Definition~\ref{def:csp}, for all finitely generated $R$-modules.
\end{prop}

\begin{proof}
For property (1) of the definition, if $y, z \in N^{F\rm sp}_M$, there is some $q$ with $y^q \in \m N^{[q]}_M$
and some $q'$ with $z^{q'} \in \m N^{[q']}_M$, and without loss of generality $q \geq q'$.  Then \[
z^q = (z^{q'})^{q/q'} \in (\m N^{[q']}_M)^{[q/q']}_{F^{e'}(M)} = \m^{[q/q']} N^{[q]}_M \subseteq \m N^{[q]}_M,
\]
and thus $(y-z)^q = y^q - z^q \in \m N^{[q]}_M$, so $y-z \in N^{F\rm sp}_M$.  Moreover, for any $r \in R$, 
$(r y)^q = r^q y^q \in \m N^{[q]}_M$, so that $r y \in N^{F\rm sp}_M$.

For property (2), $\m N \subseteq N^{F\rm sp}_M$ by taking $q=1$ in the definition, and
$N^{F\rm sp}_M \subseteq N^F_M$ because $\m N^{[q]}_M \subseteq N^{[q]}_M$.

For property (3), suppose $L \subseteq N \subseteq M$ are finitely
generated submodules and $N \subseteq (L + N^{F\rm sp}_M)^F_M$,
then since $N$ is finitely generated, there is some $q$ such that \begin{equation}\label{eq:Fspeq}
N^{[q]}_M \subseteq L^{[q]}_M + (N^{F\rm sp}_M)^{[q]}_M
\end{equation}
Since $N^{F\rm sp}_M$ is finitely generated, there is some $q'$ such that $(N^{F\rm sp}_M)^{[q']}_M \subseteq
\m N^{[q']}_M$.  Replacing the $q$ in (\ref{eq:Fspeq}) by $\max\{q,q'\}$, that containment yields \[
N^{[q]}_M \subseteq L^{[q]}_M + \m N^{[q]}_M.
\]
Then by the standard Nakayama lemma, $N^{[q]}_M \subseteq L^{[q]}_M$, which proves
that $N \subseteq L^F_M$, and hence property~(3).

Finally, for property (4), first suppose $z \in (N^F_M)^{F\rm sp}_M$.  Then there is some $q$ such that
both $z^q \in \m (N^F_M)^{[q]}$ and $(N^F_M)^{[q]}_M = N^{[q]}_M$, which combine to
make $z^q \in \m N^{[q]}_M$, and hence that $z \in N^{F\rm sp}_M$.
Similarly, if $z \in (N^{F\rm sp}_M)^F_M$, there
is some $q$ such that both $z^q \in (N^{F \rm sp}_M)^{[q]}_M$ and $(N^{F\rm sp}_M)^{[q]}_M \subseteq \m N^{[q]}_M$,
which combine to show that $z^q \in \m N^{[q]}_M$, and hence that $z \in N^{F\rm sp}_M$.
\end{proof}

The question immediately arises in which situations we have a special Frobenius closure
decomposition:

\begin{prop}
Let $(R,\m)$ be a Noetherian local ring of characteristic $p>0$ with perfect residue field.
Then for any finitely generated $R$-module $M$ and any submodule $N \subseteq M$, 
$N^F_M = N + \spF{N}_M$.
\end{prop}

\begin{proof}
The containment `$\supseteq$' is obvious.  So suppose that $z \in N^F_M$.
Then there is some $q$ such that $z^q \in N^{[q]}_M$.

Let $\{z_1, \dotsc, z_n \}$ be any generating set of $N$.  Then we have \[
z^q = \sum_{i=1}^n a_i z_i^q,
\]
where $a_i \in R$.  Let $r$ be the number of $a_i$'s that are not in $\m$. 
We can rearrange the $z_i$'s in such a way that
$a_i \notin \m$ if $1 \leq i \leq r$ and $a_i \in \m$ if $r < i \leq n$.
Since $a_i \notin \m$ for $1 \leq i \leq r$ and $R/\m$ is perfect,
there exist $u_i \in R \setminus \m$ and $m_i \in \m$ such that $a_i = u_i^q + m_i$
whenever $i \leq r$.  Hence, \[
\left( z - \sum_{i=1}^r u_i z_i \right)^q = \sum_{i=1}^r m_i z_i^q + \sum_{i=r+1}^n a_i z_i^q 
\in \m N^{[q]}_M.
\]
That is, $z - \sum_{i=1}^r u_i z_i \in \spF{N}_M$, so that $z \in N + \spF{N}_M$.
\end{proof}

By analyzing the proof of \cite[Theorem 5.1]{nme*spread}, if $(R,\m)$ is a local ring, $c$ is any closure operation on (submodules of) a class of modules $\mathcal{M}$ with a special part $c\rm{sp}$, and if for all $R$-modules $M \in \mathcal{M}$ and submodules $L \subseteq M$ one has $L^c = L + L^{c\rm{sp}}$, then submodules of $M \in \mathcal{M}$ have spread, in the sense that every minimal $c$-reduction has the same minimal number of generators as every other.  Hence, if $k$ is a perfect field, then $F$-spread is well-defined for submodules of any finitely generated $R$-module.  The assumption on the field can be dropped too.

However, we present below a different way to prove that $F$-spread is well-defined, using notions analogous to the original definitions of analytic spread and 
analytic independence from Northcott and Rees \cite{NR} for Frobenius closure, inspired also
in part by Adela Vraciu's work on $*$-independence in \cite{Vr*ind}:

\begin{defn}
Fix a finitely-generated $R$-module $M$, as before.  Let $z_1, \dotsc, z_n \in N$, 
where $N$ is a submodule of $M$.  Then we say
that $z_1, \dotsc, z_n$ are \emph{analytically $F$-independent in $N$} [resp. 
\emph{analytically $F$-independent}] if for any power $q$ of $p$ and any 
polynomial $\phi$ of the form \[
\phi(X_1, \dotsc, X_n) = c_1 X_1^q + \cdots + c_n X_n^q,
\]
where $q$ is a power of $p$, the $X_i$ are indeterminates, and the $c_i$ are elements of $R$,
such that $\phi(z_1, \dotsc, z_n) \in \m N^{[q]}_M$ [resp. such that
$\phi(z_1, \dotsc, z_n) = 0 \in N^{[q]}_M$], it follows that the coefficients
$c_1, \dotsc, c_n$ of $\phi$ are all in $\m$.
\end{defn}

\begin{lemma}
For elements $z_1, \dotsc, z_n \in M$, with $L = (z_1, \dotsc, z_n)$, the 
following are equivalent: 
\begin{enumerate}
\item $z_1, \dotsc, z_n$ are analytically $F$-independent.
\item $z_1, \dotsc, z_n$ are analytically $F$-independent in $L$.
\item $z_1, \dotsc, z_n$ are analytically $F$-independent in any submodule $N$ of $M$ such that $L$
is an $F$-reduction of $N$.
\item $z_1, \dotsc, z_n$ are $F$-independent.
\item For any power $q$ of $p$, $z_1^q, \dotsc, z_n^q$ form a minimal set of
generators for $L^{[q]}_M$.
\end{enumerate}
\end{lemma}

\begin{proof}
$(1) \Rightarrow (2):$ Let $\phi = c_1 X_1^q + \cdots + c_n X_n^q$ such that
$\phi(z_1, \dotsc, z_n) \in \m (z_1, \dotsc, z_n)^{[q]}_M$.  Then there exist
$m_1, \dotsc, m_n \in \m$ such that \[
c_1 z_1^q + \cdots + c_n z_n^q = m_1 z_1^q + \cdots + m_n z_n^q.
\]
Let $\psi = (c_1 - m_1) X_1^q + \cdots +  (c_n - m_n) X_n^q$.  Then
$\psi(z_1, \dotsc, z_n) = 0$, so that by analytic $F$-independence, $c_i - m_i \in \m$ 
for all $i$.  Thus, since $m_i \in \m$ for all $i$, it follows that $c_i \in \m$
for all $i$.

$(2) \Rightarrow (3):$ Let $q'$ be a power of $p$ such that
$N^{[q']}_M = (z_1, \dotsc, z_n)^{[q']}_M$, and let $\phi = c_1 X_1^q + \cdots + c_n X_n^q$
be such that $\phi(z_1, \dotsc, z_n) \in \m N^{[q]}_M$.  If $q \geq q'$, then 
$N^{[q]}_M = (z_1, \dotsc, z_n)^{[q]}_M$, so that by analytic $F$-independence in $(z_1, \dotsc, z_n)$,
all the $c_i$ are in $\m$.  On the other hand, if $q < q'$, then we have \begin{align*}
\phi(z_1, \dotsc, z_n)^{q'/q}_M &\in \m^{[q'/q]} N^{[q']}_M \\
&= \m^{[q'/q]} (z_1, \dotsc, z_n)^{[q']}_M \\
&\subseteq \m (z_1, \dotsc, z_n)^{[q']}_M.
\end{align*}
Thus by analytic $F$-independence in $(z_1, \dotsc, z_n)$, we have that $c_i^{q'/q} \in \m$ for
each $i$.  But $\m$ is radical, so $c_i \in \m$ for all $i$.

$(3) \Rightarrow (1):$ Obvious, since $0 \in \m N^{[q]}_M$ for all $N$ and all $q$.

$(1) \Rightarrow (4):$ If $z_1, \dotsc, z_n$ are not $F$-independent, then without 
loss of generality $z_1 \in (z_2, \dotsc, z_n)^F_M$.
Then there is some $q$ with $z_1^q \in (z_2, \dotsc, z_n)^{[q]}_M$.  Thus, there are
choices $c_2, \dotsc, c_n \in R$ such that \[
z_1^q + c_2 z_2^q + \cdots + c_n z_n^q.
\]
If we set $\phi = X_1^q + c_2 X_2^q + \cdots + c_n X_n^q$, then $\phi(z_1, \dotsc, z_n) = 0$
but not all of the coefficients of $\phi$ are in $\m$ (since the coefficient for $X_1^q$ is 1),
which shows that $z_1, \dotsc, z_n$ are not analytically $F$-independent.

$(4) \Rightarrow (1):$ The proof that (1) implies (4) can pretty much be reversed:
If $z_1, \dotsc, z_n$ are not analytically $F$-independent, then there is some polynomial
$\phi = c_1 X_1^q + \cdots + c_n X_n^q$ such that $\phi(z_1, \dotsc, z_n) = 0$ and
at least one of the $c_i$ is a unit.  We may assume that $i=1$.  Then \[
z_1^q = -c_1^{-1} (c_2 z_2^q + \cdots + c_n z_n^q),
\]
so that $z_1 \in (z_2, \dotsc, z_n)^F_M$, and $z_1, \dotsc, z_n$ are not $F$-independent.

$(2) \Leftrightarrow (5):$ The elements $z_1, \dotsc, z_n$ are $F$-independent in $L$ 
if and only if for any power $q$ of $p$, whenever
$c_1 z_1^q + \dotsc + c_n z_n^q \in \m (z_1^q, \dotsc, z_n^q) = \m L^{[q]}_M$, it follows that 
every $c_i \in \m$.  This is in turn equivalent to the statement that for any power $q$ of $p$,
$z_1^q, \dotsc, z_n^q$ is a minimal generating set of $L^{[q]}_M$.
\end{proof}

\begin{lemma}
Let $z_1, \dotsc, z_n \in M$ be $F$-independent.  Then the module $L$ that
they generate is strongly $F$-independent.
\end{lemma}

\begin{proof}
Let $y_1, \dotsc, y_n$ be another minimal set of generators.  Then the vector 
$\left(\begin{matrix} y_1 \\ \vdots \\ y_n \end{matrix}\right)$ may be obtained
by multiplying the vector $\left(\begin{matrix} z_1 \\ \vdots \\ z_n \end{matrix}
\right)$ by an invertible $n \times n$ matrix of elements of $R$.  Arguing as in
Vraciu \cite{Vr*ind}, we may reduce to the case where $y_1 = z_1 + d z_2$ and $y_i = z_i$
for all $i \geq 2$.  Here $d$ is some element of $R$.

Now, it is clear $y_i \notin (y_1, \dotsc, y_{i-1}, y_{i+1}, \dotsc, y_n)^F_M$ as long 
as $i \geq 3$, for in those cases $y_i = z_i$ and the module for which we claim its 
non-membership is $(z_1, \dotsc, z_{i-1}, z_{i+1}, \dotsc, z_n)^F_M$.

Next, suppose that $y_1 \in (y_2, \dotsc, y_n)^F_M$.  Then for some $c \in R$, \[
z_1^q + (d^q + c) z_2^q = (z_1 + d z_2)^q + c z_2^q \in (z_3, \dotsc, z_n)^{[q]}_M.
\]
Hence, $z_1^q \in (z_2, \dotsc, z_n)^{[q]}_M$, contradicting the fact that 
the $z_i$ are $F$-independent.

Finally suppose that $y_2 \in (y_1, y_3, \dotsc, y_n)^F_M$.  Then for some $q$ and some $c \in R$, \[
c z_1^q + (1 + c d^q) z_2^q \in (z_3, \dotsc, z_n)^{[q]}_M.
\]
If $c$ is a unit, this implies that $z_1^q \in (z_2, \dotsc, z_n)^{[q]}_M$, which 
is a contradiction.  If $c$ is not a unit, then $(1 + c d^q)$ is a unit, which 
implies that $z_2^q \in (z_1, z_3, \dotsc, z_n)^{[q]}_M$, also a contradiction.

Hence, $y_1, \dotsc, y_n$ are $F$-independent elements, as was to be shown.
\end{proof}

\begin{prop}\label{prop:F-spread}
Let $N$ be any submodule of $M$.  Then for any minimal $F$-reduction $L$ of $N$, the minimal number
of generators of $L$ is equal to the eventual minimal number of generators of the modules
$N^{[q]}_M$ for large enough choices of the power $q$ of $p$.  Hence, Frobenius closure has spread.
\end{prop}

\begin{proof}
Let $z_1, \dotsc, z_t$ be a minimal set of generators for $L$.  Since $z_1, \dotsc, z_t$
are $F$-independent, then for any power $q$ of $p$, $z_1^q, \dotsc, z_t^q$
form a minimal set of generators for $L^{[q]}_M$.  On the other hand, for sufficiently large
$q$, $L^{[q]}_M = N^{[q]}_M$.  Hence the minimal number of generators of such an $N^{[q]}_M$ is
always equal to the minimal number of generators of $L$.
\end{proof}

%
%

 \section{Special part of integral closure}
 \emph{Note:} The paper \cite{EH-cc} generalizes some results of this section (e.g. Proposition~\ref{prop:Reesvalsp}) though the point of view is very different from the one adopted here, in several respects.

For background on integral closure of ideals, the author recommends the recent book \cite{HuSw-book} of Huneke and Swanson, and in particular Chapter 10 on Rees valuations.
 
\begin{defn}
For an ideal $I$ in a Noetherian local ring $(R,\m)$, define the \emph{special part 
of the integral closure of }$I$ to be the set \[
\spintcl{I} := \{x \in R \mid \exists n \in \N \text{ such that } x^n \in \overline{\m I^n}\}.
\]
\end{defn}

\begin{prop} Let $(R,\m)$ be a Noetherian local ring, and $J \subseteq I$ ideals in $R$.
Then $\spintcl{I}$ is an ideal, $\spintcl{J} \subseteq \spintcl{I}$, and
if $R$ has prime characteristic $p>0$, then $\sptc{I} \subseteq \spintcl{I}$.
\end{prop}

\begin{proof}
For the first statement, let $x,y \in \spintcl{I}$ and $a \in R$.  
It is obvious from the definition that $a x \in \spintcl{I}$.
So we only need to show that $x + y \in \spintcl{I}$.

There exist positive integers $r$ and $s$ such that $x^r \in \intcl{(\m I^r)}$ and $y^s \in \intcl{(\m I^s)}$.
Let $n = r s$.  Then \[
x^n = (x^r)^s \in \left(\intcl{(\m I^r)}\right)^s \subseteq \intcl{(\m^s I^{r s})} \subseteq \intcl{(\m I^n)},
\]
and by symmetry we also have $y^n \in \intcl{(\m I^n)}$.  So it suffices to show that
if $x^n, y^n \in \intcl{(\m I^n)}$, then $(x+y)^n \in \intcl{(\m I^n)}$.  Since integral
closure may be computed modulo minimal primes, we may assume from this point on that
$R$ is an integral domain.

Now, by one of the equivalent
definitions for integral closure in integral domains, there is some $c \neq 0$ such that for all positive integers $t$,
\begin{equation}\label{eq:cnt}
c x^{n t}, c y^{n t} \in (\m I^n)^t.
\end{equation}
Note also the general fact that arises from looking at monomials that for any nonnegative integers $n$ and $t$:
\begin{equation}\label{eq:mon}
(x+y)^{n(t+1)} \in (x+y)^n \left(x^n,y^n\right)^t.
\end{equation}
Let $d = c^2 (x+y)^n$.  Clearly $d \neq 0$, and we have \begin{align*}
d \left((x+y)^n\right)^t &= c^2 (x+y)^{n(t+1)} & \\
&\in c^2 \left(x^n, y^n \right)^t, &\text{by (\ref{eq:mon})} \\
&= \sum_{j=0}^t (c x^{n j}) (c y^{n(t-j)}) & \\
&\subseteq \sum_{j=0}^t (\m I^n)^j (\m I^n)^{t-j} &\text{by (\ref{eq:cnt})} \\
&= (\m I^n)^t.&
\end{align*}
Hence, $(x+y)^n \in \intcl{(\m I^n)}$, as was to be shown.

It is obvious that $\spintcl{J} \subseteq \spintcl{I}$.

The third statement follows from the fact that  \[
(\m I^{[q]})^* \subseteq (\m I^q)^* \subseteq \overline{\m I^q}.
\]
for all powers $q = p^e$ of $p$.

\end{proof}

We need the following symbols, following Samuel (op. cit.): \begin{itemize}
\item $v(I) := \inf \{v(x) \mid x \in I \}$ and
\item $v_I(x) := \sup \{n \in \N \cup \infty \mid x \in I^n \}$
\end{itemize}

First note that for any commutative ring $R$, any $(\mathbb{R}_{\geq 0} \cup \infty)$-valued valuation
$v$ defined on $R$, and any proper ideal $J$ of $R$, we have $v(\bar{J}) = v(J)$.

\begin{proof}
Since $J \subseteq \bar{J}$, $v(\bar{J}) \leq v(J)$.  On the other hand, let $x \in \bar{J}$.
Then there is some $k$ such that $x^{n+k} \in J^n (J,x)^k \subseteq J^n$ for all $n \in \N$.  Hence, \[
(n+k) v(x) = v(x^{n+k}) \geq v(J^n) = n v(J),
\]
so that $v(x) \geq \frac{n}{n+k} \cdot v(J)$ for all $n \in \N$.  It follows that
$v(x) \geq v(J)$, whence $v(\bar{J}) \geq v(J)$.
\end{proof}

\begin{prop}\label{prop:Reesvalsp}
Let $I$ be an ideal of $R$ and let $v_1, \dotsc, v_t$ be the Rees valuations
of $I$, with centers $\p_1, \dotsc, \p_t$ respectively.  Let $\q = \p_1 \cap \cdots 
\cap \p_t$ be their intersection.  Then the following are equivalent for
any $x \in R$: \begin{enumerate}
\item There is some $n_0 \in \N$ such that $x^n \in I^{n+1}$ for all $n \geq n_0$.
\item There is some $n \in \N$ such that $x^n \in I^{n+1}$.
\item There is some $r \in \N$ such that $x^r \in \intcl{(I^{r+1})}$.
\item There is some $n \in \N$ such that $x^n \in \q I^n$.
\item There is some $n \in \N$ such that $x^n \in \intcl{(\q I^n)}$.
\item $v_i(x) > v_i(I)$ for $1 \leq i \leq t$.
\item $x \in \spintcl{(I R_{\p_i})}$ for $1 \leq i \leq t$.
\end{enumerate}
In particular, if $I$ is $\m$-primary, then $x \in \spintcl{I}$ iff 
$x^n \in I^{n+1}$ for some $n$ iff $v(x) > v(I)$ for all Rees valuations
$v$ of $I$.
\end{prop}

\begin{proof}
It is obvious that $(1) \Rightarrow (2)$, $(2) \Rightarrow (3)$ (taking $n=r$) and $(4) \Rightarrow (5)$.  Also,
$(2) \Rightarrow (4)$ because $I \subseteq \p_i$ for $1 \leq i \leq t$, which
implies that $I \subseteq \q$.

\noindent $(3) \Rightarrow (2)$: There is some integer $n_0$ such that
for all positive integers $k$, \[
(x^r)^{n_0 + k} \in (I^{r+1})^k.
\]
In particular, letting $k = n_0 r + 1$ and $n = n_0 r^2 + n_0 r + r$, we have \[
x^n = x^{n_0 r^2 + n_0 r + r} = (x^r)^{n_0 + k} \in (I^{r+1})^k = I^{n_0 r^2 + n_0 r + r + 1} = I^{n+1}.
\]

\noindent $(5) \Rightarrow (6)$: Suppose $x^n \in \overline{\q I^n}$, and let
$v = v_i$ for some $1 \leq i \leq t$.  Then we have: \[
n v(x) = v(x^n) \geq v(\overline{\q I^n}) = v(\q I^n) 
\geq v(\p_i I^n) = v(\p_i) + n v(I) > n v(I).
\]
Hence, $v_i(x) > v_i(I)$ for $1 \leq i \leq t$.\footnote{We have proved
a bit more here, actually.  In particular, \[
v_i(x) - v_i(I) \geq \frac{v_i(\p_i)}{n}.
\]}

\noindent $(6) \Rightarrow (1)$: By the Rees valuation theorem \cite[Theorem 4.16]{Rees-book}, \[
\lim_{n \rightarrow \infty} \frac{v_I(x^n)}{n} = \min_{1 \leq i \leq t} 
\frac{v_i(x)}{v_i(I)} > 1.
\]
So there is some $n_0 \in \N$ such that $\frac{v_I(x_n)}{n} > 1$ for all
$n \geq n_0$.  Hence for such $n$, $v_I(x^n) > n$, whence since $v_I$ 
is integer-valued, $v_I(x^n) \geq n+1$, which means that
$x^n \in I^{n+1}$.

\noindent $(7) \Rightarrow (6)$ is 
clear from the definitions.  $(5) \Rightarrow (7)$
is because integral closure is persistent and $\q \subseteq \p_i$ for $1 \leq i \leq t$.

The last statement follows from the fact that if $I$ is $\m$-primary 
then each of its Rees valuations has center $\m$.
\end{proof}

At this point, the reader may protest that we haven't yet shown that 
$-\rm sp$ is a special part of the integral closure operation.  That situation
will soon be remedied, but first we note the following important lemma of Lipman's
from Huneke's paper:

\begin{lemma}\cite[Lemma 3.4]{Hu3D}\label{lem:Lipman}
Let $R$ be a Noetherian local integral domain, let $I$ be an ideal of $R$, let
$K$ be the quotient field of $R$, and let
$x \in R$.  Then if $x$ is in $I V$ for each discrete valuation ring $V$ between
$R$ and $K$ whose center on $R$ is $\m$, then $x \in \intcl{I}$.
\end{lemma}

Next, note the following `asymptotic' property associated with the definition of $-\rm sp$.

\begin{lemma}\label{lem:spbarasymptotic}
Let $(R,\m)$ be a Noetherian local ring, $I \subseteq R$ a proper ideal, $x \in R$, and $n_0 \in \N_+$.
If $x^{n_0} \in \overline{\m I^{n_0}}$, then $x^n \in \overline{\m I^n}$ for all $n \geq n_0$.
\end{lemma}

\begin{proof}
First assume that $R$ is an integral domain.  Let $V$ be any discrete valuation ring
between $R$ and the quotient field $K$ of $R$ whose center on $R$ is $\m$, and let
$v$ be its associated discrete valuation on $K$.  Then we
have \[
n_0 v(x) = v(x^{n_0}) \geq v(\overline{\m I^{n_0}}) = v(\m) + n_0 v(I)
\]
so that $v(x) \geq \frac{v(\m)}{n_0} + v(I)$.  Then for any $n \geq n_0$, we have
$v(x) \geq \frac{v(\m)}{n} + v(I)$, so that \[
v(x^n) = n v(x) \geq v(\m) + n v(I) = v(\m I^n).
\]
That is, $x^n \in \m I^n V$ for all such $V$.  Hence, by Lemma~\ref{lem:Lipman},
$x^n \in \overline{\m I^n}$.

If we drop the assumption that $R$ is a domain, the result follows immediately from
the fact that integral closure in $R$ can be computed modulo the minimal primes
of $R$.
\end{proof}

\begin{lemma}
$x \in \spintcl{I}$ if and only if this is true modulo all minimal primes of $R$.
\end{lemma}

\begin{proof}
This follows immediately from the above lemma and the fact that the corresponding 
statement is true for integral closure.
\end{proof}

\begin{prop}\label{pr:-sp}
$-\rm sp$ is a special part of integral closure, in the sense
of Definition~\ref{def:csp}.
\end{prop}

\begin{proof}
We already showed that $\spintcl{I}$ is an ideal (i.e. axiom 1 of the definition), and it is clear
from the definitions that $\m I \subseteq \spintcl{I} \subseteq \intcl{I}$ (axiom 2).

Now for (3), suppose that $J \subseteq I \subseteq (J + \spintcl{I})^-$. 
First, we may assume without loss of generality that $I$ is integrally closed.
Next, recall that when we were proving that $\spintcl{I}$ is an ideal, we showed
that for any $n$ and any $x, y \in R$, if $x^n, y^n \in \intcl{(\m I^n)}$ then
$(x+y)^n \in \intcl{(\m I^n)}$.  It follows easily from this fact along with
Lemma~\ref{lem:spbarasymptotic} that there is some $n$ such that for any 
$x \in \spintcl{I}$, $x^n \in \intcl{(\m I^n)}$.  Thus, if $\mu = \mu(I)$, then \[
(\spintcl{I})^{\mu n} = (\spintcl{I})^{n} (\spintcl{I})^{n (\mu - 1)} \subseteq \intcl{(\m I^n)} I^{n(\mu - 1)}
\subseteq \intcl{(\m I^{\mu n})}
\]

There is some $r$ with $I^{r+1} \subseteq (J + \spintcl{I}) I^r$.  Then letting
$\mu = \mu(I)$ and $m = \mu n$, \begin{align*}
(I^m)^{2 r + 2} &= (I^{r+1})^{2 m} \subseteq (J^m + (\spintcl{I})^m)(J + \spintcl{I})^m I^{2 r m} \\
&\subseteq (J^m + \intcl{(\m I^m)})(I^m)^{2 r + 1}
\end{align*}

Now, after modding out by a minimal prime, we may assume
that $R$ is a domain.  Let $v$ be any $\m$-centered valuation.  Then \[
(2 r + 2) v(I^m) = v((I^m)^{2 r + 2}) \geq v(J^m + \intcl{(\m I^m)}) + (2 r + 1) v(I^m),
\]
so that \[
v(I^m) \geq \min\{ v(J^m), v(\intcl{(\m I^m)}) \} \geq \min \{ v(J^m), 1 + v(I^m) \}.
\]
Thus, $m v(I) = v(I^m) \geq v(J^m) = m v(J)$, which 
means that $v(I) \geq v(J)$.  Since this holds for all $\m$-centered valuations $v$, it follows
from Lemma~\ref{lem:Lipman} that $I \subseteq \intcl{J}$. 

Finally we prove (axiom 4) for integral closure.  Note first that for any minimal prime $\p$, \[
\frac{\spintcl{I} + \p}{\p} = \spintcl{\left(\frac{I + \p}{\p}\right)}
\]
and \[
\frac{\intcl{I} + \p}{\p} = \intcl{\left(\frac{I + \p}{\p}\right)}.
\]
Hence, if it holds for integral domains, then we have for any
minimal prime $\p$ of $R$: \[
\frac{\intcl{(\spintcl{I})} + \p}{\p} = \intcl{\left(\spintcl{\left(\frac{I + \p}{\p} \right)} \right)}
 = \spintcl{\left( \frac{I + \p}{\p} \right)} = \frac{\spintcl{(\intcl{I})} + \p}{\p}.
\]
Thus (axiom 4) for integral closure holds in $R$.  So we may assume from now on that $R$ is an integral domain.

Suppose $x \in \spintcl{(\intcl{I})}$.  Then for some positive integer $n$, we 
have $x^n \in \intcl{\left(\m \left(\overline{I}\right)^n\right)}$.  Hence, by Lemma~\ref{lem:Lipman},
for any valuation $v$ on $K$ centered on $\m$ in $R$, where $K$ is the fraction field of $R$, we have: \[
v(x^n) \geq v\left(\m \left(\overline{I}\right)^n\right) = v(\m) + n v(\bar{I}) = v(\m) + n v(I) = v(\m I^n).
\]
Hence, by Lemma~\ref{lem:Lipman} again, $x^n \in \overline{(\m I^n)}$, so that $x \in \spintcl{I}$.

Now let $x \in \intcl{(\spintcl{I})}$.  Then there is some integer $r$ and some elements 
$a_i \in (\spintcl{I})^i$ for $1 \leq i \leq r$ such that \[
x^r = \sum_{i=1}^r a_i x^{r-i}.
\]
Take any valuation $v$ of $K$ centered on $\m$ in $R$.  Then \[
r v(x) = v(x^r) \geq \min \{v(a_i) + (r-i) v(x) \mid 1 \leq i \leq r \}.
\]
In particular, for each $v$ there exists some $i$ between $1$ and $r$ (dependent on $v$)
such that $r v(x) \geq v(a_i) + (r-i) v(x)$.  Hence, \[
v(x) \geq \frac{v(a_i)}{i}.
\]
But
there exists some $t$ such that for all $j$, $a_j^t \in \intcl{(\m I^{j t})}$, so that in
particular, \[
v(a_i) \geq \frac{v(\m)}{t} + i v(I).
\]
Combining the latest two displayed equations, we have \[
v(x^{r t}) = r t v(x) \geq \frac{r v(\m)}{i} + r t v(I) \geq v(\m) + r t v(I) = v(\m I^{r t}),
\]
since $r \geq i$.  Noting that $r$ and $t$ are independent of the choice of $v$, Lemma~\ref{lem:Lipman}
then implies that $x^{r t} \in \intcl{(\m I^{r t})}$, so that $x \in \spintcl{I}$.
\end{proof}

 \section{Evolutions and Fermat's Last Theorem}
After Eisenbud and Mazur \cite{EisMaz} connected ``evolutionary stability'' and the Wiles-Taylor proof
of Fermat's last theorem with symbolic squares, H\"ubl \cite{H:ueval} related their
methods to certain questions about integral closure of ideals, as well as the fiber
cone of and associated graded ring to an ideal.

In particular, he showed that if $k$ is a field of characteristic 0, and $S$ is a reduced
local algebra essentially of finite type over $k$, then $S/k$ is evolutionarily stable
if and only if it has a presentation $S = R/I$, $R/k$ smooth, such that $(R,I)$ satisfies
the following condition ``(NN)'': \[
I \cap \{f \in R \mid \exists n \text{ such that } f^n \in I^{n+1} \} = \m I.
\]

In section 3 of his paper, H\"ubl considers the following conditions on a ring and an ideal
$(R,I)$.  (MR) says that
if $f \in I \setminus \m I$, then $f$ is contained in some minimal reduction of $I$.
Condition (AR) says that \[
I \cap \{f \in R \mid \exists n \text{ such that } x^n \in \m I^n \} = \m I.
\]
Call the following condition (SP): \[
I \cap \spintcl{I} = \m I.
\]
Clearly (MR) $\Rightarrow$ (SP) $\Rightarrow$ (AR) $\Rightarrow$ (NN), with none of the
arrows reversible.  Moreover, since $-\rm sp$ is in fact a special part of integral closure (Proposition~\ref{pr:-sp}),
if follows from Lemma~\ref{lem:capmL}
that whenever $I$ is bar-independent, it satisfies (SP), hence also (NN).

Thus, if $R$ is a regular local ring essentially of finite type over a field $k$
of characteristic 0, and $I$ is a radical bar-independent ideal (e.g. $I$ may be
a radical ideal with no proper reductions), $R/I$ is an evolutionarily stable
algebra over $k$.

 \section{Special tight closure Brian\c{c}on-Skoda theorems}

The history of ``Brian\c{c}on-Skoda theorems'' goes back more than 35 years and 
could itself be the subject of a short essay.  The original theorem, proved in 1974
by Brian\c{c}on and Skoda is as follows (with notation slightly altered):

\begin{thm}\label{thm:BS}\cite[Th{\'e}or{\`e}me 3]{BriSkoda} 
Let $I$ be an $n$-generated ideal in the convergent power series ring $R=\C\{z_1, \dotsc, z_d\}$.  Then if $t=\inf\{n, d\}$, $(I^-)^t \subseteq I$.
\end{thm}
 
An algebraic proof, which generalized the theorem
to \emph{all} regular local rings, was given in 1981 by Lipman and Sathaye:

\begin{thm}\label{thm:BS-LipSath}\cite[special case of Theorem 1]{LipSath}
Let $I$ be an ideal in a regular local ring $R$, let $\ell$ be the analytic spread of $I$, and let $w \geq 0$ be an integer.  Then $\intcl{(I^{\ell+w})} \subseteq I^{w+1}$.
\end{thm}
This is a generalization because the analytic spread of an ideal
is bounded above by both the number of generators of the ideal and the dimension of the ring.

In 1990, Hochster and Huneke gave a tight closure proof, generalizing the Brian\c{c}on-Skoda Theorem to \emph{all} rings of characteristic $p$ (and later, for rings of equal characteristic zero, after tight closure was well-defined for such rings), but not including the mixed characteristic case:

\begin{thm}\label{thm:BS-HH}(\cite[Theorem 5.6]{HHmain} and \cite[Theorem 4.1.5]{HH-tcz})
Let $I$ be an ideal in a Noetherian local ring $R$ of equal characteristic, let $I$ be an ideal of analytic spread $\ell$, and let $w \geq 0$ be an integer.  Then $\intcl{(I^{\ell+w})} \subseteq (I^{w+1})^*$.
\end{thm}

The reason why this generalizes
Theorem~\ref{thm:BS} (and Theorem~\ref{thm:BS-LipSath} when $R$ contains a field) is that
$(I^{w+1})^* = I^{w+1}$ if $R$ is regular.  Theorem~\ref{thm:BS-HH} is really a theorem about the tight closure of an ideal capturing the integral closure
of a not-much-higher power of that ideal.  It is noteworthy that although Theorem~\ref{thm:BS-LipSath} has a very difficult proof, Theorem~\ref{thm:BS-HH} (at least in characteristic $p$) is extremely easy once the foundations
of tight closure theory are laid down.

Similarly, we can prove ``special'' versions, as follows:

\begin{prop}[Special tight closure Brian\c{c}on-Skoda theorem]\label{prop:spbc}
Let $(R,\m)$ be a Noetherian local ring of characteristic $p>0$, and $I$ a proper ideal
of $R$.  If $n = \mu(I)$ and $w$ is any nonnegative integer, then \[
\spintcl{(I^{n+w})} \subseteq \sptc{(I^{w+1})}
\]
\end{prop}

\begin{proof}
Without loss of generality $R$ is an integral domain, since the special parts of both 
the integral and tight closures can be computed modulo minimal primes.

Suppose $0 \neq x \in \spintcl{(I^{n+w})}$, where $n = \mu(I)$.  Then by Lemma~\ref{lem:spbarasymptotic},
there is some power $q_1$ of $p$ with $q_1 \geq \mu(I)$ such that 
$x^{q_1} \in \overline{\m I^{q_1 (n+w)}}$.  Let $q_0$ be a power of $p$ such that $q_0 \geq \mu(\m)$.
Then there exists some integer
$k$ such that for all powers $q$ of $p$, \begin{align*}
x^{q_1 k} (x^{q_1 q_0})^{q} &= x^{q_1(k + q q_0)} \in \left(\m I^{(n+w)q_1}\right)^{q q_0} \\
&\subseteq \m^{q q_0} \left( I^{n+w} \right)^{q q_1 q_0} 
 \subseteq \left(\m^{[q]}\right)^{q_0 - \mu(\m) + 1} \left(I^{[q q_1 q_0]}\right)^{w+1} \\
&\subseteq \left( \m \left( I^{w+1} \right)^{[q_1 q_0]}\right)^{[q]}.
\end{align*}
Hence, $x^{q_1 q_0} \in \left(\m (I^{w+1})^{[q_1 q_0]}\right)^*$, which shows
that $x \in \sptc{(I^{w+1})}$.
\end{proof}

\begin{cor}[Special Brian\c{c}on-Skoda theorem in characteristic $p$]\label{cor:spbc}
Let $(R,\m)$ be a Noetherian regular (or weakly $F$-regular) local ring
of characteristic $p>0$, $n = \mu(I)$, and $w$ any nonnegative
integer.  Then \[
\spintcl{\left( I^{n+w}\right)} \subseteq \m I^{w+1}.
\]
\end{cor}
\begin{proof}
This follows from Proposition~\ref{prop:spbc} and the fact that in a weakly $F$-regular
local ring $(R,\m)$, $\sptc{J} = \m J$ for any proper ideal $J$ of $R$.
\end{proof}

It would be interesting to prove the above corollary in the equicharacteristic zero
case as well, by reduction to characteristic $p$, or even 
in mixed characteristic (perhaps using methods of Lipman, Sathaye, Teissier, et. al.).

 \section{The special part of the integral closure of monomial ideals}

For a standard graded ring $S$ over a field $k$, there is a unique homogeneous maximal ideal $\m$, and we may define the special part of the integral closure of a homogeneous ideal $J$ in analogous fashion, namely let $\spintcl{J}$ be the ideal generated by all homogeneous elements $x$ of $S$ such that for some integer $t$, $x^t \in \intcl{(\m J^t)}$.  Then one can show (routinely) that all homogeneous elements of $\spintcl{J}$.  Moreover:

\begin{lemma}
Let $S$ be a standard $\N$-graded Noetherian domain over a field $k$, with irrelevant maximal ideal $\m$.  Let $J$ be a homogeneous ideal of $S$, and let $n$ be the lowest degree
among degrees of elements generating $J$.  Then $\spintcl{J}$ contains no homogeneous elements
of degree less than or equal to $n$.
\end{lemma}
\begin{proof}
Let $x$ be a homogeneous element of $\spintcl{J}$, and let $d$ be its degree.  Then there is some integer $t>0$ with $x^t \in \intcl{(\m J^t)}$.  
Hence there is some positive integer $k$ such that \[
x^{t k} \in \m I^t \left(\m J^t, x^t\right)^{k-1}.
\]
The expression on the left-hand side has degree $d t k$.  On the other hand, any element 
of the expression on the right-hand side has degree greater than or equal to \[
1 + n t + (k-1) \min\{1 + n t, d t \}.
\]
So if $d \leq n$, then \[
d t k \geq 1 + n t + (k-1) (d t) \geq 1 + k d t,
\]
which is a contradiction.
\end{proof}

\noindent \textit{Convention: } For the rest of this section, we will fix a polynomial ring $R = k[x_1, \dotsc, x_n]$ in $n$ variables, using the
standard $\N$-grading, with $k$ a field.  Let $\m$ be the homogeneous maximal ideal $(x_1, \dotsc, x_n)$.

Note that if $f \in I$, where $I$ is a monomial ideal, then if we express 
$f = c_1 m_1 + \cdots + c_r m_r$ where the $m_i$ 
are monomials in the variables $x_j$ and the $c_i \in k$ (any such $f$ has a unique such
expression, of course), then $m_i \in I$ for all $i$.  This is due to the $\N^n$-(multi)graded nature of the polynomial ring $R$.  We will use this
fact repeatedly, sometimes without mentioning it.

It is folk knowledge (see, e.g. \cite[Exercises 4.22-23]{Eis-CAbook}) that if we let
$\Gamma(I)$ denote the set of exponents (as elements of $\N^n$) of the monomials in a
monomial ideal $I$, then \[
\Gamma(\intcl{I}) = \func{conv}(\Gamma(I)) \cap \N^n,
\]
where ``conv'' denotes the convex hull of a subset of $\R^n$.

Another way of expressing this
set is as follows.  Let $\{ x^{\beta_1}, \dotsc, x^{\beta_r} \}$ be a minimal set
of generators of $I$.  (Here we use double subscripting, so that $\beta_{i,j}$ is the exponent of
$x_j$ in the monomial $m_i$.)  Then for a monomial $x^\alpha$,
$\alpha \in \Gamma(\intcl{I})$ if and only if there exist nonnegative rational numbers $c_1, \dotsc, c_r$
such that \begin{equation}\label{eq:intmon}
\sum_{i=1}^r c_i = 1 \quad \text{and} \quad \alpha \geq \sum_{i=1}^r c_i \beta_i.
\end{equation}
The partial ordering on $\R^n$ we use is the standard one, where $\gamma \geq \delta$ if
$\gamma_i \geq \delta_i$ for all $i$, and $\gamma > \delta$ means both that $\gamma \geq \delta$
and that $\gamma \neq \delta$.

With this latter characterization of integral closure of a monomial ideal, we are ready
to describe the special part of the integral closure in similar terms.

\begin{prop}\label{prop:spintmon}
Let $R = k[x_1, \dotsc, x_n]$, where $k$ is a field, and $\m = (x_1, \dotsc, x_n)$.
Let $I$ be a monomial ideal of $R$ contained in $\m$, minimally generated by
monomials $\{x^{\beta_1}, \dotsc, x^{\beta_r} \}$.  Then $\spintcl{I}$ is also a monomial
ideal, and for a monomial $x^\alpha$, $\alpha \in \Gamma(\spintcl{I})$ if and only if
there exist nonnegative rational numbers $c_1, \dotsc, c_r$ such that \begin{equation}\label{eq:spintmon}
\sum_{i=1}^r c_i = 1 \quad \text{and} \quad \alpha > \sum_{i=1}^r c_i \beta_i.
\end{equation}
\end{prop}

\begin{proof}
First we show that $\spintcl{I}$ is a monomial ideal.  Let $f = b_1 m_1 + \cdots 
+ b_u m_u \in \spintcl{I}$, where the $m_i$ are distinct monomials and $0 \neq b_i \in k$
for all $i$.  Then there is some positive integer $t$ such that $f^t \in \overline{\m I^t}$.
In particular, since the latter is a monomial ideal, $m_i^t \in \overline{\m I^t}$ for
$i = 1, \dotsc, r$, which means that $m_i \in \spintcl{I}$ for each $i$.

Now, let $\alpha \in \N^n$, and suppose that $x^\alpha \in \spintcl{I}$.  Then there is
some $t$ with $x^{t \alpha} \in \overline{\m I^t}$.  Then by Vitulli \cite[Corollary 3.2]{Vitnormal}, there is some $s$ with
$x^{s t \alpha} \in \m^s I^{s t}$.  In particular, there exists a positive integer $q$
with $x^{q \alpha} \in \m I^q$.  This means that $x^{q \alpha}$ is a multiple of one of the
generating monomials of $\m I^q$.  In particular, there exist nonnegative integers
$p_1, \dotsc, p_r$ such that \begin{equation}\label{eq:spintmonint}
\sum_{i=1}^r p_i = q \quad \text{and} \quad q \alpha > p_1 \beta_1 + \cdots + p_r \beta_r.
\end{equation}
(The ``$>$'' is because $x^{q \alpha} \in \m I^q$, and not merely in $I^q$.)  Then dividing
through by $q$ and letting $c_i = p_i / q$, we get (\ref{eq:spintmon}).

Conversely, suppose that $\alpha$ and $c_1, \dotsc, c_r$ satisfy (\ref{eq:spintmon}).  Since the
$c_i$ are rational, they have a common denominator, say $q$, so that there are nonnegative
integers $p_1, \dotsc, p_r$ such that $c_i = p_i / q$ for each $i$, satisfying (\ref{eq:spintmonint}).
Hence $x^{q \alpha} \in \m I^q$, which means that $x^\alpha \in \spintcl{I}$.
\end{proof}

We use this to tell us exactly when special decomposition of integral closure fails
for monomial ideals.  First, for any subset $C$ of $\R^n$, let \[
\func{low}(C) := \{P \in C \mid Q \not < P \text{ for all } Q \in C \},
\]
the ``lowest points'' of $C$.

\begin{cor}
Let $R = k[x_1, \dotsc, x_n]$, where $k$ is a field, and $\m = (x_1, \dotsc, x_n)$.
Let $I$ be an ideal of $R$ minimally generated by distinct nontrivial
monomials $\{x^{\beta_1}, \dotsc, x^{\beta_r} \}$.  Let $S = \{ \beta_1, \dotsc, \beta_r \}$.
Then $\Gamma(\intcl{I})$ is the disjoint union of $\Gamma(\spintcl{I})$ with 
$\func{low}(\func{conv}(S)) \cap \N^n$.  Hence, 
$\intcl{I} = I + \spintcl{I}$ if and only if $S = \func{low}(\func{conv}(S)) \cap \N^n$.
\end{cor}

For example, if $I = (x^t, y^t)$, we have $\spintcl{I} = \m I$, but $\intcl{I} = (x,y)^t$,
so the decomposition fails if $t > 1$.  In general, if $I = (x_1^{p_1}, \dotsc, x_n^{p_n})$
for integers $p_i$,
then the decomposition holds if and only if whenever $1 \leq i < j \leq n$, $\func{gcd}(p_i,p_j) = 1$.

 \section{Intersections and compatibility}

\begin{lemma}
Let $I$ be a proper ideal in the local ring $(R,\m)$ of characteristic $p>0$.  Then 
$\spintcl{I} \cap I^* = \sptc{I}$.
\end{lemma}
\begin{proof}
Let $f \in \spintcl{I} \cap I^*$.  Then there is some $q_1$ with $f^{q_1} \in \overline{\m I^{q_1}}$,
and some $q_0$ and some $c \in R^o$ with $c f^q \in I^{[q]}$ for all $q \geq q_0$.
Thus, there is some $d \in R^o$ such that for all powers $q, q_2$ of $p$, 
$d f^{q q_1 q_2} \in \m^{q q_2} I^{q q_1 q_2}$.  Then we have: \begin{align*}
c d (f^{q_1 q_2})^q = c d f^{q q_1 q_2} &\in \m^{q q_2} I^{q q_1 q_2} \cap I^{[q q_1 q_2]} \\
&\subseteq \m^{q q_2} (I^{q_2 - \mu(I) + 1})^{[q q_1]} \cap I^{[q q_1 q_2]} \\
&\subseteq \m^{q q_2} I^{[q q_1]} \cap I^{[q q_1 q_2]} \\
&\subseteq \m^{q q_2 - r} I^{[q q_1 q_2]} \\
&\subseteq \m^{[q]} I^{[q q_1 q_2]} = (\m I^{[q_1 q_2]})^{[q]}
\end{align*}
Thus, $f^{q_1 q_2} \in (\m I^{[q_1 q_2]})^*$, so $f \in \sptc{I}$.
\end{proof}

\begin{lemma}
If $(R,\m)$ is an excellent analytically irreducible local ring and $I$ 
is a $*$-independent ideal in $R$, then $\sptc{I} \cap I^F = \spF{I}$.
\end{lemma}
\begin{proof}
Let $f \in \sptc{I} \cap I^F$, and let $f_1, \dotsc, f_d$ be a $*$-independent
generating set of $I$.  Then there is some power $q_1$ of $p$ such that
$f^{q_1} \in I^{[q_1]}$.  Hence \[
f^{q_1} = \sum_{i=1}^d a_i f_i^{q_1}.
\]
Also, there is some $q_0$ such that $c f^q \in \m^{[q/q_0]} I^{[q]}$ for all $q \gg 0$.
That is, \[
c f^q = \sum_{i=1}^d m_{i q} f_i^q, 
\]
where $m_{i q} \in \m^{[q / q_0]}$
for all such $q$.  On the other hand, from the first displayed equation we also have \[
c f^q = \sum_{i=1}^d c a_i^{q / q_1} f_i^q.
\]
Combining the previous two displayed equations, we have \[
\sum_{i=1}^d (c a_i^{q / q_1} - m_{i q}) f_i^q = 0.
\]
Since $f_1, \dotsc, f_d$ are $*$-independent and the colon criterion \cite[Proposition 2.4]{AbflatFreg} holds in $R$,
there is some power $q_2 \geq \max \{ q_0, q_1 \}$ of $p$ such that
$c a_i^{q / q_1} - m_{i q} \in \m^{[q / q_2]}$, so that $c a_i^{q / q_1} \in \m^{[q/q_2]}$
for all $q \gg 0$ and all $1 \leq i \leq d$.  Hence $a_i^{q_2/q_1} \in \m^* = \m$,
which implies that $a_i \in \m$.  So we have $f^{q_1} \in \m I^{[q_1]}$, 
whence $f \in \spF{I}$.
\end{proof}

\section*{Acknowledgements}
The author wishes to thank Professor Craig Huneke for constant support and advice during the development of these results.

\end{document}